\documentclass[12pt]{amsart}

\allowdisplaybreaks[1]
\usepackage{enumerate}
\usepackage{graphicx}
\usepackage{allrunes}
\usepackage{amsmath}
\usepackage{amssymb}
\usepackage{amsfonts}
\usepackage[dvipsnames]{xcolor}
\usepackage{hyperref}
\usepackage{skak}
\allowdisplaybreaks
\makeindex

 \newtheorem{theorem}{Theorem}[section]
 
 \newtheorem{lemma}[theorem]{Lemma}
 \newtheorem{proposition}[theorem]{Proposition}

\theoremstyle{definition}
\newtheorem{definition}[theorem]{Definition}
\theoremstyle{remark}

\newtheorem{fact*}{Fact}

\newcommand{\C}{\mathbb{C}}

\renewcommand{\R}{\mathbb{R}}

\newcommand{\ad}{^\ast}
\newcommand{\inv}{^{-1}}

\newcommand{\til}{\raise.17ex\hbox{$\scriptstyle\mathtt{\sim}$}}

\newcommand\beq{\begin{equation}}

\newcommand\eeq{\end{equation}}

\newcommand{\bbm}{\left[ \begin{smallmatrix}}
\newcommand{\ebm}{\end{smallmatrix} \right]}
\newcommand{\bpm}{\left( \begin{smallmatrix}}
\newcommand{\epm}{\end{smallmatrix} \right)}
\numberwithin{equation}{section}

\newlength{\Mheight}
\newlength{\cwidth}

\newcommand{\dfn}[1]{{\bf #1}\index{#1}}

\title[The noncommutative Schoenberg theorem]{Noncommutative Schur-type products and their Schoenberg theorem}
\author[J. E. Pascoe]{
J. E. Pascoe
}
\address{Department of Mathematics\\
1400 Stadium Rd\\
  University of Florida\\
 Gainesville, FL 32611}
\email[J. E. Pascoe]{pascoej@ufl.edu}

\date{\today}


\begin{document}

\begin{abstract}
Schoenberg showed that a function $f:(-1,1)\rightarrow \R$ such that $C=[c_{ij}]_{i,j}$ positive semi-definite implies that $f(C)=[f(c_{ij})]_{i,j}$ is also positive semi-definite 
must be analytic and have Taylor series coefficients nonnegative at the origin. The Schoenberg theorem is essentially a theorem about the functional calculus arising from the Schur product,
the entrywise product of matrices.
Two important properties of the Schur product are that the product of two rank one matrices is rank one, and the product of two positive semi-definite matrices is positive semi-definite.
We classify all products which satisfy these two properties and show that these generalized Schur products satisfy a Schoenberg type theorem.
\end{abstract}

\maketitle


\section{The classical case}
	The Schur product of two matrices is given by their entrywise product.
	Two important properties are that the Schur product of two positive semi-definite matrices is again positive semi-definite,
	and that the Schur product of two rank one matrices is again rank one. The natural functional calculus arising from the Schur product
	is entrywise evaluation. That is, given a function $f:(a,b)\rightarrow \mathbb{C},$ we apply the function $f$ to a matrix $C=[c_{ij}]_{i,j}$
	with entries in $(a,b)$
	by defining $f(C)=[f(c_{ij})]_{i,j}.$
	Schoenberg proved the following result, modulo some mild refinements due to Rudin.
	\begin{theorem}[Schoenberg \cite{schoenberg1942}, Rudin \cite{rudin1959}]
		Let  $f:(-1,1)\rightarrow \mathbb{R}.$ If 
		$$C\geq 0 \Rightarrow f(C)\geq 0,$$
		then $f(x) = \sum c_kx^k$ for some coefficients $c_k\geq 0.$
		
		Here, $C \geq 0$ means that the matrix $C$ is positive semi-definite.
	\end{theorem}
	Much has been made of the Schoenberg theorem in recent years, see \cite{Belton2019,Belton2019a} for an extensive survey, including several multi-variable generalizations \cite{fitzy}. Our present goal will be to
	characterize products which preserve positivity and rank-ones, and moreover show that they satisfy a Schoenberg type theorem.
\section{Rank-one preserving products}
	We begin by characterizing rank-one preserving products.
	\begin{definition}
		Let $*$ be an associative bilinear product on the space of $n$ by $m$ matrices over $\C$
		with an identity element which has rank one.
		\begin{enumerate}
			\item We say $*$ is a \dfn{rank-one preserving product} or \dfn{ropp} whenever
			the product of two rank one matrices is rank one or less.
			\item	We say $*$ is a  \dfn{semi-definite and rank-one preserving product}
			or \dfn{saropp} whenever $*$ is a ropp and the product of
			two positive semi-definite matrices is again positive semi-definite.
 		\end{enumerate}
	\end{definition}
	The Schur product on $n$ by $m$ matrices is an example of a ropp, and when $n=m$
	the Schur product is a saropp.

	Another class of examples of ropps are the generalized Schur products.
	\begin{definition}	
	Let
	$\mathcal{A}$ be a unital algebra of dimension $n$ and 
	$\mathcal{B}$ be a unital algebra of dimension $m.$
	Fix linear bijections $v_{\mathcal{A}}: \mathcal{A} \rightarrow \C^n$
	and $v_{\mathcal{B}}: \mathcal{B} \rightarrow \C^m.$
	We define a \dfn{generalized Schur product} on $n$ by $m$ matrices
	to be the unique bilinear product satisfying the relation
		$$v_{\mathcal{A}}(a)v_{\mathcal{B}}(b)\ad * 
		v_{\mathcal{A}}(c)v_{\mathcal{B}}(d)\ad
		= v_{\mathcal{A}}(ac)v_{\mathcal{B}}(bd)\ad.$$
	\end{definition}
	When $\mathcal{A}=\mathcal{B}$ and $v_{\mathcal{A}}=v_{\mathcal{B}},$
	the corresponding generalized Schur product is a saropp. 
	Whenever $\mathcal{A} = \C^n$ and $\mathcal{B} = \C^m,$ both endowed with the entry-wise product, we recover the classical Schur product
	as a generalized Schur product. One of the main goals of the current section is to show that 
	all ropps and saropps arise in this way respectively. We will drop the rather outlandish ropp / saropp terminology once that fact is established.

	We note that, for a generalized Schur product, there is a natural isomorphism,
		$$(M_{nm},*) \cong \mathcal{A} \otimes \overline{\mathcal{B}}$$
	where complex conjugation on $\mathcal{B}$ is induced from $v_{\mathcal{B}}.$
	(Concretely, $ab^* \mapsto v_{\mathcal{A}}\inv(a)\otimes v_{\mathcal{B}}\inv(\overline{b})\ad .$)

%
%
%
	\subsection{Ropps are generalized Schur products}
	\begin{proposition} \label{commutes}
		Let $*$ be a ropp on $n$ by $m$ matrices. Write the identity element for $*$ as $11^*.$
		Then,
		$$v1^* * 1w^*= 1w^* * v1^* = vw^*.$$
	\end{proposition}
	\begin{proof}
		Without loss of generality, $v$ is not in the span of $1$ and  $w$ is not in the span of $1^*.$ 
		Write $v1^* * 1w^* = \nu \omega^*.$
		Note that
		\begin{align*} 
			(a1+bv)1^**1(c1+dw)^*&=ac11^* * 11^*+ ad11^* * 1w^*+\\&\phantom{===}bcv1^* * 11^*+ bdv1^* * 1w^* \\
			&=ac11^* + ad 1w^*+bcv1^*+ bd\nu\omega^*
		\end{align*}
		must be rank $1$ or less for all real choices of $a, b, c, d.$
		
		An elementary argument then says that $\nu\omega^*=vw^*.$ 
		Namely, we know that $\nu\omega^*$ must either have common range or common kernel with $1w^*$ by setting $c=0.$
		So, either we know that we can take $\omega = w,$ or $\omega = 1.$
		In the case where $\omega = w,$ $\nu$ must be a multiple of $v$ as $\nu\omega^*$ must either have common range or common kernel with $1v^*.$ Taking $a=b=c=d=1$ then gives that
		$\nu\omega^*=vw^*.$ In the case where $\omega = 1,$ $\nu$ must be a multiple of $1$ as $\nu\omega^*$ must either have common range or common kernel with $1v^*.$ Taking $a=b=c=d=1$ then 
		a matrix with rank $2,$ contradicting our hypotheses.
		So, we are done.
	\end{proof}

	\begin{proposition} \label{product}
		Let $*$ be a ropp on $n$ by $m$ matrices. Write the identity element for $*$ as $11^*.$
		Then, for all $v, w,$ there exists a $u$ such that
		$$v1^* * w1^*= u1^*.$$
		Similarly,
		$$1v^* * 1w^*= 1u^*.$$
	\end{proposition}
	\begin{proof}
		We will prove the first identity, the second identity is similar.

		Without loss of generality, $v, w$ are not in the span of $1.$
		Write $v1^* * w1^* = \nu \omega^*.$
		Note that
		\begin{align*} 
			(a1+bv)1^**(c1+dw)1^*&=ac11^* * 11^*+ ad11^* * w1^*+\\&\phantom{===}bcv1^* * 11^*+ bdv1^* * w1^* \\
			&=ac11^* + (adw+bcv)1^* + bd\nu\omega^*
		\end{align*}
		must be rank $1$ or less for all real choices of $a, b, c, d.$

		An elementary argument then says that $\omega^*$ was in the span of $1^*.$ (If the product was non-zero.)
		Namely, we know that $\nu\omega^*$ must either have common range or common kernel with $w1^*$ by setting $c=0.$	
		If $\omega^*$ is not in the span of $1^*,$ then we can take $\nu = w.$
		Moreover, $\nu$ is also a multiple of $v.$ Without loss of generality $v=w.$
		So, we have that $ac11^* + (ad+bc)v1^* + bdv\omega^*$ is rank one.
		So, taking $ad+bc =0,$ we see that $ac11^* + bdv\omega^*$ is rank $1$ and so $\omega^*$ must be a multiple of $1^*.$
		So, we are done.
	\end{proof}

	\begin{theorem}
		If $*$ is a ropp on $n$ by $m$ matrices, then $*$ is a generalized Schur product.
	\end{theorem}
	\begin{proof}
		 Write the identity element for $*$ as $11^*.$
		Define the algebra $\mathcal{A} = \C^n$ to be an algebra with product such that $(vw)1^* = v1^**w1^*.$ (This is well defined by Proposition \ref{product}.)
		Similarly, let $\mathcal{B}=\C^m$ be endowed with the product $1(vw)^* = 1v^**1w^*.$
		Using $\mathcal{A}$ and $\mathcal{B}$ and $v_\mathcal{A}, v_\mathcal{B}$ being the identity gives that $*$ was a generalized Schur product.
		That is, applying Proposition \ref{commutes}, we see that
		\begin{align*} 
			v_{\mathcal{A}}\inv(a)v_{\mathcal{B}}\inv(b)\ad * 
			v_{\mathcal{A}}\inv(c)v_{\mathcal{B}}\inv(d)\ad &
			= ab^* * cd^*\\
			&= a1^* * 1b^* * c1^* * 1d^*\\
			&= a1^* * c1^*  * 1b^* *  1d^*\\
			&= (ac)1^* * 1(bd)^* \\
			&= v_{\mathcal{A}}\inv(ac)v_{\mathcal{B}}\inv(bd)\ad.
		\end{align*}
	\end{proof}
\section{The Schoenberg theorem}
	First we define the natural domain for the noncommutative Schoenberg theorem.
	Given a generalized Schur product $*$ on the set of $n$ by $m$ matrices, we define $\mathcal{S}_*$ to be the algebra of $n$ by $m$ matrices equipped with $*.$
	We define the \dfn{$d$-dimensional generalized Schur universe} to be 
		$$\mathcal{S}^d = \bigsqcup \mathcal{S}_*^d$$
	where the disjoint union is taken over all generalized Schur products $*.$
	We define the \dfn{$N$ by $N$ matrix generalized Schur universe} to be 
		$$\mathcal{S}^{N \times N} = \bigsqcup M_N(\mathbb{C}) \otimes \mathcal{S}_*.$$
	\begin{definition}
		Let $*$ be a generalized Schur product on $n$ by $m$ matrices.
		We say a $d$-tuple of $n$ by $m$ matrices $(X_1, \ldots, X_d)$ is a \dfn{Schur spectral contraction with respect to $*$} whenever
		there are $A, B$ such that the tuple
			$$\bbm A & X^* \\ X & B  \ebm$$
		is positive definite and has joint spectral radius less than $1.$
		(Here, the block matrix sits in $m+n$ by $m+n$ matrices endowed with generalized Schur product with algebra $\mathcal{A}\oplus \mathcal{B},$ and the joint spectral radius is
		taken with respect to that product.)

		We denote the set of Schur spectral contractions in $d$ variables by $\mathcal{C}^d.$

		Since the joint spectral radius is continuous \cite{strangcontinuity}, the set of Schur spectral contractions is open.
	\end{definition}
	
	Now we define our class of functions.
	\begin{definition}
		Let $\mathcal{U} \subseteq \mathcal{S}^d.$
		We define a \dfn{scalar noncommutative function} $f:\mathcal{U} \rightarrow \mathcal{S}$ satisfying the following axioms.
		\begin{enumerate}
			\item $f(X)=f(X_1,\ldots, X_d)$ is contained in the algebra generated by $X_1, \ldots, X_d.$
			\item If there is a homomorphism $\varphi$ from the algebra generated by $(X_1, \ldots, X_d)\in\mathcal{U}$ to the algebra generated by 
			$(Y_1, \ldots, Y_d)\in \mathcal{U}$ such that $\varphi(X_i)=Y_i,$ then $f(Y) = \varphi(f(X)).$ 
		\end{enumerate} 
		We say a function $f:\mathcal{U} \rightarrow \mathcal{S}^{N \times N}$ is a \dfn{noncommutative function} if each of its block entries is a scalar noncommutative function.

		We note that this implies that
			$$f\left(\bbm A & B \\ C & D  \ebm\right) = \bbm f(A) & f(B) \\ f(C) & f(D)  \ebm$$
		when a block matrix is given a generalized Schur product corresponding to a direct sum on both sides and the block matrix is in the domain.

		We say a map $f:\mathcal U \rightarrow \mathcal{S}^{N \times N}$ is \dfn{positivity preserving} whenever $X \in \mathcal S^d_* \cap \mathcal{U},$ where $*$ is positivity preserving and $X \geq 0,$
 		implies that $f(X) \geq 0.$
	\end{definition}

	\begin{lemma}
		A positivity preserving map is locally bounded on the set of Schur spectral contractions.
	\end{lemma}
	\begin{proof}
		Suppose $X$ is a Schur spectral contraction. By definition, there are $A, B$ such that the tuple
			$$\bbm A & X^* \\ X & B  \ebm$$
		is positive definite and has joint spectral radius less than $1.$ 
		Therefore, by continuity of the joint spectral radius and the smallest eigenvalue, for any small $H,$ $X+H$ we also have that 
		$$\bbm A & (X+H)^* \\ X+H & B  \ebm$$
		is positive definite and has joint spectral radius less than $1.$
		Therefore, $$f(\bbm A & (X+H)^* \\ X+H & B  \ebm) = \bbm f(A) & f((X+H)^*) \\ f(X+H) & f(B)  \ebm \geq 0.$$
		So, $$\|f(X+H)\| \leq \max(\|f(A)\|,\|f(B)\|).$$
	\end{proof}
	
	The following theorem is omnipresent in the noncommutative function theory literature, see \cite{vvw12} for a treatment in ordinary setting. We prove the same holds in the functional calculus
	arising from generalized Schur products.
	\begin{lemma}\label{differentiable}
		If $f$ is a positivity preserving map on the set of Schur spectral contractions, then $f$ is differentiable.
		
		Hence, since the function is defined on an open complex domain, the function is analytic.
	\end{lemma}
	\begin{proof}
		Let $X$ be a Schur spectral contraction.
		By definition, there are $A, B$ such that the tuple
			$$Y=\bbm A & X^* \\ X & B  \ebm$$
		is positive definite and has joint spectral radius less than $1.$
		Fix a direction $K$ and find 
			$$H=\bbm G & K^* \\ K & F  \ebm$$
		which is positive definite.
		We will show that $f$ is differentiable at $Y,$ and thus at $X$ by showing that the third difference quotient is always positive for positive directions $H.$
		That is, without loss of generality, we will show that $f(Y+3H)-3f(Y+2H)+3f(Y+H)-f(X)\geq 0$ whenever $Y, H \geq 0$ and the quantity is well defined. (A function whose third difference quotient is always positive 
		must be differentiable.)

		We will show the first difference quotient is positive, and then an inductive argument essentially proves the claim. Consider 
			$$f\left(\bbm 
			Y+H & Y \\ Y &Y+H
			 \ebm\right)= \bbm 
			f(Y+H) & f(Y) \\ f(Y) &f(Y+H)
			 \ebm\geq 0.$$
		Evaluating on the block vector $\bbm 1/\sqrt 2  \\ -1/\sqrt 2 \ebm,$ we see that $f(X+H)-f(X) \geq 0,$ so we are done. (Essentially, the function $g(X,H)=f(X+H)-f(X)$ is again positivity preserving,
		so in fact an arbitrary difference quotient is positive. This induces a function on the domain of points $(X,H)\in \mathcal{S}^{2d}$ such that $X+H \in \mathcal{C}^d.$) Repeating this process eventually
		proves the claim.
	\end{proof}
	We now state and prove our noncommutative generalization of Schoenberg's theorem.
	\begin{theorem}[The noncommutative Schoenberg theorem]
		Let $f$ be a noncommutative function on the set of Schur spectral contractions which is positivity preserving.
		Then, $f$ has a noncommutative power series representation $f(X) = \sum c_\alpha X^\alpha$
		which converges for all Schur spectral contractions, and each $c_\alpha \geq 0.$

		Here, the power series is evaluated with the underlying generalized Schur product $*.$ 
	\end{theorem}
	\begin{proof}
		Since $f$ is analytic on each algebra by Lemma \ref{differentiable} and $f(X)$ is contained in the algebra generated by the coordinates of $X,$ it is clear that 
		$f$ has a noncommutative power series $f(X) = \sum c_\alpha X^\alpha$ which converges absolutely on the domain. (The set of Schur spectral contractions is star-like and balanced with respect to the origin.)
		So, it suffices to show that the $c_\alpha \geq 0.$ Let $\mathcal{A}$ be the algebra of truncated noncommutative polynomials of degree $d.$ (That is, noncommutative polynomials quotiented out by the monomials
		of degree greater than $d.$) Find a spatial isomorphism $v_\mathcal{A}$ from $\mathcal{A}$ into $\mathbb{C}^N$ for some large $N.$ Let $x_1, \ldots, x_d$ be the coordinate functions in $\mathcal{A}.$
		Now, $f(v_{\mathcal{A}}(x)v_{\mathcal{A}}(x)^*) =  \sum_{|\alpha|\leq d} c_\alpha v_{\mathcal{A}}(x^\alpha)v_{\mathcal{A}}(x^\alpha)^* \geq 0$ and we are done.
	\end{proof}

\bibliography{references}
\bibliographystyle{plain}

\printindex

\end{document}